\documentclass[12pt]{elsarticle}

\usepackage{amssymb, amsthm, amsmath}
\usepackage[all]{xy} \input xy
\usepackage{xy} \input xy \xyoption{all}
\usepackage{graphicx}
\usepackage{soul}   
\usepackage[noend]{algorithmic}
\algsetup{linenodelimiter=.}
\usepackage[usenames,dvipsnames]{color} 
\usepackage{hyperref} 
\hypersetup{colorlinks, citecolor=blue, filecolor=blue, linkcolor=blue, urlcolor=blue} 

\theoremstyle{plain}
\newtheorem{theorem}{Theorem}[section]

\theoremstyle{definition}
\newtheorem{definition}[theorem]{Definition}
\newtheorem{example}[theorem]{Example}
\newtheorem{remark}[theorem]{Remark}
\newtheorem{algorithm}[theorem]{Algorithm}

\numberwithin{equation}{section} 

\newcounter{mnotecounter}

\newcommand{\F}{\mathbb{F}}
\newcommand{\C}{\mathbb{C}}

\newcommand{\N}{\mathbb{N}}

\newcommand{\mB}{\mathcal{B}}
\newcommand{\mV}{\mathcal{V}}
\newcommand{\mP}{\mathcal{P}}
\newcommand{\mR}{\mathcal{R}}

\newcommand{\mA}{\mathcal{A}}
\newcommand{\mT}{\mathcal{T}}
\newcommand{\ot}{\,\overline{t}\,}
\newcommand{\odelta}{\,\overline{\delta}\,}
\newcommand{\ox}{\,\overline{x}\,}
\newcommand{\OT}{\,\overline{T}\,}
\newcommand{\ODelta}{\,\overline{\Delta}\,}
\newcommand{\OX}{\,\overline{X}\,}

\newcommand{\BP}{\mB_\mP}
\newcommand{\VP}{\mV_\mP}
\newcommand{\AP}{\mA_\mP}

\newcommand{\DP}{\mathcal{D}_\mP}
\newcommand{\VT}{\mV_{\mathbb{T}}}
\newcommand{\ii}{\mathtt{i}}   

\DeclareMathOperator{\Ima}{Im}
\DeclareMathOperator{\Dom}{Dom}
\DeclareMathOperator{\genus}{genus}

\newcommand{\menos}{\smallsetminus}

\begin{document}

\title{Algebraic and algorithmic aspects of radical parametrizations\tnoteref{t1}}
\tnotetext[t1]{Partially supported by the Spanish \emph{Ministerio de Econom\'ia y Competitividad} under Project MTM2014-54141-P, and by Junta de Extremadura and FEDER funds.}

\author[uah]{J. Rafael Sendra} \ead{Rafael.Sendra@uah.es}
\author[uex]{David Sevilla\corref{c1}\fnref{group-uex}} \ead{sevillad@unex.es}
\author[uah]{Carlos Villarino} \ead{Carlos.Villarino@uah.es}

\fntext[group-uex]{Member of the research group GADAC (Ref. FQM024).}
\cortext[c1]{Corresponding author}

\address[uah]{Research group ASYNACS \\ Dept. of Physics and Mathematics, University of Alcal\'a \\ E-28871 Alcal\'a de Henares (Madrid, Spain)}
\address[uex]{Campus of M\'erida, University of Extremadura \\ Av. Santa Teresa de Jornet 38 \\ E-06800 M\'erida (Badajoz, Spain)}

\begin{abstract}
In this article algebraic constructions are introduced in order to study the variety defined by a radical parametrization (a tuple of functions involving complex numbers, $n$ variables, the four field operations and radical extractions). We provide algorithms to implicitize radical parametrizations and to check whether a radical parametrization can be reparametrized into a rational parametrization.
\end{abstract}


\maketitle

\section{Introduction}

It is well known that in many applications dealing with geometric objects, parametric representations are very useful (see \cite{HoschekLasser1993a}). In general, when one works with parametric representations of algebraic curves and surfaces, the involved functions are rational. Nevertheless it is well known that only genus 0 curves, and arithmetic and plurigenus 0 surfaces, have this property. Furthermore, even if we limit our parametrizations to the rational case, many of the  geometric constructions in CAGD, as e.g. offsetting, conchoidal or cissoid constructions, do not propagate the rationality of the geometric object. This means that even though the original object is rational the new object is not in general.

One possible way to overcome this limitation is to work with piecewise approximate parametrizations. A second way is to extend the family of functions  used in the parametric representation; for instance using radicals of polynomials. The latter is the frame of this paper. When dealing with curves, algorithms to parametrize with radicals can be found in \cite{SendraSevilla2011a} and \cite{Harrison2013a} that cover the cases of genus less or equal to 6. In addition, in \cite{SendraSevilla2013a} one can find algorithms to parametrize by radicals certain classes of surfaces. An additional interesting property of radical parametrizations is that the radical nature of a variety is preserved under geometric constructions of degree up to 4 (see Section 5 in \cite{SendraSevilla2013a} for further details). So, in particular, offsets and  conchoids of radical varieties are radical (see Corollaries 5.2 and 5.3. in \cite{SendraSevilla2013a}).

In this article we continue the exploration of radical parametrizations initiated in \cite{SendraSevilla2011a} and \cite{SendraSevilla2013a}. Here we introduce a framework to manipulate these parametrizations in a rational way by means of rational auxiliary varieties and maps. This allows us to apply results of algebraic geometry to derive conclusions on the radical parametrization and its image. In short, to translate radical statements into rational ones. More precisely, from the theoretical point of view, we introduce the notion  of radical variety associated to a radical parametrization, and we prove that it is irreducible and of dimension equal to the number of parameters in the parametrization. Furthermore, we introduce the notion of tracing index of radical parametrization that extends the notion of properness of rational parametrization (see \cite{SendraWinkler2001a}). In addition, we define an algebraic variety, that we call tower variety, which is birationally equivalent to the radical variety when the tracing index is 1. The most interesting property of the tower variety is that it encodes rationally the information of the radical parametrization. From the algorithmic point of view, we show how to compute generators of the radical variety and of the tower variety, in particular an implicitization algorithm for radical parametrizations, and how to compute the tracing index. As a potential application we present an algorithm to decide, and actually compute, whether a given radical parametrization can be reparametrized into a rational parametrization. Also, we show how the tower variety may help to compute symbolically integrals whose integrand is a rational function of radicals of polynomials.

The paper is structured as follows: in Section~\ref{SECTION: notion of radparam} we recall the notion of radical parametrization and we discuss how to represent them.
In Section \ref{SECTION: radical variety} we introduce the concept of radical variety and we prove some of the main properties. In Section \ref{SECTION: rational vs radical} the tower variety is defined, properties are presented, and its application to check the reparametrizability of radical parametrizations into rational parametrizations is illustrated.

\section{The notion of radical parametrization}\label{SECTION: notion of radparam}

A radical parametrization is, intuitively speaking, a tuple \linebreak $\ox=(x_1(\ot),\ldots,x_r(\overline{t}))$ of functions of variables $\overline{t}=(t_1,\ldots,t_n)$ which are constructed by repeated application of sums, differences, products, quotients and roots of any index; we will assume in the sequel that $r>n$. More formally, a radical parametrization is a tuple of elements of a radical extension of the field $\C(\overline{t})$ of rational functions in the variables $\overline{t}$. In the following we approach the concept by means of Field Theory.

\begin{definition}\label{DEF: radical tower and radical param}
A \emph{radical tower over $\C(\ot)$} is a tower of field extensions $\F_0=\C(\ot)\subseteq\dots\subseteq\F_{m-1}\subseteq\F_{m}$ such that $\F_i=\F_{i-1}(\delta_i)=\F_0(\delta_1,\ldots,\delta_i)$ with $\delta_i^{e_i}=\alpha_i\in\F_{i-1}$, $e_i\in\N$. In particular, $\C(\ot)$ is a radical tower over itself.
\end{definition}

\begin{definition}\label{DEF: rad param}
A \emph{radical parametrization} is a tuple $\ox(\ot)$ of elements of the last field $\F_m$ of some radical tower over $\C(\ot)$, such that their Jacobian has rank $n$.
\end{definition}

\begin{remark} $ $
	\begin{enumerate}
		\item A rational parametrization is a radical parametrization.
		\item The Jacobian is defined by extension of the canonical derivations $\frac{\partial}{\partial t_i}$ from $\C(\ot)$ to $\F_m$ (see \cite[Chapter 8.5, Theorem 5.1]{Lang2002a}, or \cite[Chapter II.17, Corollary 2]{ZariskiSamuel1975a}  for the formal details). One can calculate the derivatives of the $\delta$'s recursively as follows: for each expression $\delta_i^{e_i}=\alpha_i$ in the definition of the tower we write a relation $\Delta_i^{e_i}=\alpha_i(\Delta_1,\ldots,\Delta_{i-1})$ where $\Delta_1,\ldots,\Delta_i$ are new variables dependent on the $\ot$. Then we can differentiate with respect to any $t_i$ to obtain
		\[
		e_i\Delta_i^{e_i-1}\frac{\partial\Delta_i}{\partial t_j}=\frac{\partial\alpha_i}{\partial t_j},
		\]
		the right hand side involving $\Delta_1,\ldots,\Delta_{i-1}$ and their derivatives. Substituting the $\delta$'s into the $\Delta$'s we obtain an explicit relation between $\frac{\partial\delta_i}{\partial t_j}$ and the previous partial derivatives.
	\end{enumerate}
\end{remark}

Let us illustrate the notion of radical parametrization with an example, in order to relate the usual way of writing radical expressions to our Definition \ref{DEF: rad param}. See for instance \cite{CavinessFateman1976a} and \cite[Section 2.6]{DavenportSiretTournier1988a} for more information on the topic of representation and simplification of radical functions.
\begin{example}\label{EX: 1st radicalparam from radexpr}
	One would expect that the expression
	\begin{equation}\label{param1}
	\displaystyle\left(\frac{1}{\sqrt[6]{t}\sqrt[3]{t}-\sqrt{t}},t\right)
	\end{equation}
	is not defined because the denominator is zero. This is due to the default interpretations of the roots as the principal branches (i.e. $\sqrt[n]{1}=+1$). Let us try to be more explicit about those branches by interpreting \eqref{param1} as a parametrization in  the sense of Definition \ref{DEF: rad param}.
	
	For this purpose, we need to introduce a  radical tower over $\F_0=\C(t)$. We can consider the following tower
	\[ \mbox{$\mathbb{T} := \left[\F_0\subset \F_0(\delta_1) \subset \F_0(\delta_1,\delta_2) \subset \F_0(\delta_1,\delta_2,\delta_3),\right.$ where $\left.\delta_{1}^{2}=t, \, \delta_{2}^{3}=t,\,\delta_{3}^6=t \right]$.} \]
	Note that there are different choices for the $\delta_i$, but all possible choices of conjugates generate the same tower. Thus, we can write the parametrization as
	\begin{equation}\label{param2}
	\displaystyle\left(\frac{1}{\delta_3\delta_2-\delta_1},t\right)
	\end{equation}	
	In order to specify the radical parametrization, we will fix particular choices of the $\delta_i$. For that, one possibility is to provide a value $t_0\in\C$ as well as values $\delta_{1}(t_0),\delta_2(t_0),\delta_3(t_0)$ such that \eqref{param2} is well-defined.	
	 For instance, $\delta_1(1)=-1,\delta_2(1)=e^{2\pi \ii/3},\delta_3(1)=+1$ will produce the expression
	\begin{equation}\label{param2b}
	\displaystyle\left(\frac{e^{-\pi \ii/3}}{\sqrt{t}},t\right)
	\end{equation}
	where the latter root denotes the principal branch. Note that for this particular election the Jacobian has the required rank. Finally, observe that not all choices of the $\delta_i$ are valid. For instance, if we choose $\delta_1(1)=\delta_2(1)=\delta_3(1)=1$, then  the denominator is identically zero. Summarizing, \eqref{param1} is not a radical parametrization unless we specify the tower and we choose the $\delta_i(t)$ branches properly. A conveniently compact notation for \eqref{param2b} is
	\[
	\displaystyle\left\{\left(\frac{1}{\delta_3\delta_2-\delta_1},t\right), \mathbb{T}, \delta_1(1)=-1,\delta_2(1)=e^{2\pi i/3},\delta_3(1)=1 \right\}.
	\]
	Other towers can be used in this construction, for example  $\mathbb{T}' := [\C(t) \subset \C(t)(\delta_1) \subset \C(t,\delta_1)(\delta_2), \delta_{1}^{2} = t, \delta_{2}^{6}
= t]$ and the branch choices
$\delta_1(1) = -1, \delta_{2}(1) = 1$. Then the expression (\ref{param1}) is specified to the radical parametrization
\[ \left(\frac{1}{2\sqrt{t}},t\right) \]
where the latter root denotes the principal branch.

We will not address in this article the problems of finding optimally simple or canonical ways of expressing radical parametrizations.
\end{example}

\begin{example}
	Let us generate a radical parametrization from the expression
	\begin{equation}\label{param3}
	(x(t),y(t)) = \left( t+\frac{\sqrt[4]{t^3+2t}}{\sqrt{t^2-\sqrt[3]{t-1}}} , \frac{\sqrt[4]{5\sqrt[3]{t-1}+1}}{t^3+5} \right)
	\end{equation}
	We consider the tower
	\[ \mathbb{T} := \renewcommand{\arraystretch}{1.2}\left[\begin{array}{l} \C(t)\subset \F_1:=\C(t)(\delta_1) \subset \F_2:=\F_1(\delta_2) \subset \F_3:=\F_2(\delta_3) \subset \F_4:=\F_3(\delta_4),\\
	\mbox{where $\delta_{1}^{3}=t-1, \ \delta_{2}^{2}=t^2-\delta_1, \delta_{3}^{4}=5\delta_{1}+1, \ \delta_{4}^{4}=t^3+2t$}\end{array} \right]. \]
	Then, one possible radical parametrization from the ambiguous expression \eqref{param3} is
	\[
	\displaystyle\left\{\left(t+\frac{\delta_{4}}{\delta_2},\frac{\delta_3}{t^3+5}\right), \mathbb{T}, \delta_1(2)=1,\delta_{2}(2)=+\sqrt{3},\delta_{3}(2)=+\sqrt[4]{6},\delta_{4}(2)=+\sqrt[4]{12} \right\}.
	\]
	Another way of writing the same parametrization is
	\[
	\displaystyle\left\{\left(t+\ii\,\frac{\delta_{4}}{\delta_2},\frac{\delta_3}{t^3+5}\right), \mathbb{T}, \delta_1(2)=1,\delta_{2}(2)=-\sqrt{3},\delta_{3}(2)=+\sqrt[4]{6},\delta_{4}(2)=\ii\sqrt[4]{12} \right\}.
	\]
	Note that we have changed the definitions of $\delta_2$ and $\delta_4$ and compensated the change with a coefficient in the parametrization.
\end{example}

\section{The variety of a radical parametrization}\label{SECTION: radical variety}

In this section we associate to a radical parametrization an algebraic variety that coincides with the intuitive notion of the (Zariski closure of the) image of the vector function defined naturally from the parametrization. Furthermore we establish some of its main properties. With the notations of the previous section, let
\[
\mP=\left\{(x_1(\ot),\ldots,x_r(\ot)),\  \mathbb{T},\ \delta_{1}(\ot_0)=a_1,\ldots,\delta_{m}(\ot_0)=a_m \right\}
\]
be a radical parametrization. Note that the $\delta$'s  and the $x$'s  (see Definitions \ref{DEF: radical tower and radical param} and \ref{DEF: rad param}) are not necessarily rational over $\ot$, but they can be expressed rationally as
\begin{equation}\label{ratfunsexpr}
\begin{split}
\alpha_1=\frac{\alpha_{1N}(\ot)}{\alpha_{1D}(\ot)}, \ \alpha_2=\frac{\alpha_{2N}(\ot,\delta_1)}{\alpha_{2D}(\ot,\delta_1)}, \ \ldots, \ \alpha_m=\frac{\alpha_{mN}(\ot,\delta_1,\ldots,\delta_{m-1})}{\alpha_{mD}(\ot,\delta_1,\ldots,\delta_{m-1})}, \\ x_i(\ot)=\frac{x_{iN}(\ot,\odelta)}{x_{iD}(\ot,\odelta)} \quad \mbox{where $\odelta=(\delta_1(\ot),\ldots,\delta_m(\ot))$}.
\end{split}
\end{equation}
\begin{remark}
The rational functions in \eqref{ratfunsexpr} are not uniquely determined by the algebraic elements that define $\mP$, due to the $\delta$'s being algebraic elements over the rational function field. In the objects defined below we indicate only the dependence of $\mP$ as subindices for simplicity of notation.
\end{remark}

For our constructions we will define the maps $\varphi$ and $\pi$, whose domains will be specified later where needed. Let
\begin{equation}\label{EQ: psi and varphi}
\varphi\colon\C^n\to\C^{n+m+r}\colon\ot\mapsto(\ot,\odelta,\ox)
\end{equation}

\begin{definition}\label{DEF: incidence variety}
The \emph{incidence variety} associated to this representation of $\mP$ is $\BP=\overline{\Ima(\varphi)}\subset\C^{n+m+r}$.
\end{definition}

\begin{remark}
The notation $\BP$ should not mislead the reader into thinking that $\mP$ determines the indicence variety, since $\BP$ depends crucially on the representation of $\mP$: the tower and the functions introduced at the beginning of this section.
\end{remark}

The last map that we define is the projection from $\BP$ onto the variables of interest, the $\ox$:
\[
\pi\colon\BP\subset\C^{n+m+r}\to\C^r\colon(\ot,\odelta,\ox)\mapsto\ox.
\]
\begin{definition}\label{DEF: radical variety}
The \emph{radical variety} associated to the parametrization $\mP$ is $\VP=\overline{\pi(\BP)}\subset\C^r$.
\end{definition}

In the rest of this section we will prove some basic properties of these objects, among them that $\BP$ and $\VP$ are irreducible of dimension $n$. To this end we define an auxiliary variety, the variety $\AP$, that also has computational interest. For this purpose, we introduce the tuples of variables
\[
\OT=(T_1,\ldots,T_n),\ODelta=(\Delta_1,\ldots,\Delta_m), \,\OX=(X_1,\ldots,X_r)\,\,\, \text{and} \,\, Z.
\]
In the ring $\C[\OT,\ODelta,\OX,Z]$ we define the polynomials:
\begin{equation}\label{EQ: AP unprojected}
\renewcommand{\arraystretch}{1.5}
\begin{array}{l}
E_1 = (\Delta_1)^{e_1}\cdot\alpha_{1D}(\OT) - \alpha_{1N}(\OT), \\
E_i = (\Delta_i)^{e_i}\cdot\alpha_{iD}(\OT,\Delta_1,\ldots,\Delta_{i-1})-\alpha_{iN}(\OT,\Delta_1,\ldots,\Delta_{i-1}), \quad i=2,\ldots,m, \\
G_j = X_j\cdot x_{jD}(\OT,\ODelta) - x_{jN}(\OT,\ODelta), \qquad j=1,\ldots,r, \\
G_Z = Z\cdot \mathrm{lcm}(x_{1D},\ldots,x_{rD}, \alpha_{1D}, \cdots, \alpha_{mD}) - 1.
\end{array}
\end{equation}

Let us denote the zeroset over $\C$ of finitely many polynomials $\{P,Q,\ldots\}$ as $V(P,Q,\ldots)$.

\begin{definition}\label{DEF: aux variety}
$\AP$ is the Zariski closure of the projection of\linebreak $V(E_1,\ldots,E_m,G_1,\ldots,G_r,G_Z)$ onto all but the last coordinate (therefore $\AP\subset\C^{n+m+r}$).
\end{definition}

The objects and maps defined so far are shown in a commutative diagram.
\begin{equation}\label{eq:diagram 1}
\xy
(0,0)*{\xy
	(-5,5)*++{\BP}="BP";
	"BP"+(19,0.5)*{\subset\AP\subset\C^{n+m+r}};
	(-11,-25)*++{\VP}="VP";
	"VP"+(-5,0)*{\supset}; "VP"+(-9.5,0)*++{\C^r};		
	(20,-25)*++{\C^n}="Cn";
	{\ar_{\displaystyle \pi} "BP"; "VP"};
	{\ar_{\displaystyle \varphi} "Cn"; "BP"};
	{\ar_{\displaystyle \mP} "Cn"; "VP"};
\endxy};
(35,0)*{\mbox{defined as}};
(70,0)*{\xy
	(0,5)*+{(\ot,\odelta,\ox)}="BP";
	(-11,-25)*++{\ox}="VP";
	(20,-25)*++{\ot}="Cn";
	{\ar@{|->}_{\displaystyle \pi} "BP"; "VP"};
	{\ar@{|->}_{\displaystyle \varphi} "Cn"; "BP"};
	{\ar@{|->}_{\displaystyle \mP} "Cn"; "VP"};
\endxy};
\endxy
\end{equation}

Two examples will serve here as an illustration.

\begin{example}
Consider the radical parametrization $(t,\sqrt{1-t^2})$ of the unit circle, where the root denotes the principal branch. An expression for this is
\[
\mP = \left\{\left(t,\delta\right), \ \mathbb{T} := [\C(t)\subset\C(t)(\delta), \delta^2=1-t^2], \ \delta(0)=+1 \right\}.
\]	
Then:
\begin{itemize}
	\item $\AP$ is $V(\Delta^2-1+T^2,X-T,Y-\Delta)$ (the $Z$-equation was $Z-1$, trivially eliminated). This is a circle in $\C^4$.
	\item $\BP$ is the closure of the image of $\varphi(t)=(t,\sqrt{1-t^2},t,\sqrt{1-t^2})$ which coincides with $\AP$.
	\item $\VP$ is the projection of $\BP$ onto the variables $X,Y$ which is the circle $X^2+Y^2=1$.
\end{itemize}
\end{example}

\begin{example}\label{EX: superficie}
This example is based on Example 4.8 in \cite{SendraSevilla2013a}. Consider the radical parametrization
$$
\left(t_{{2}},{\frac {t_{{2}}\left(\sqrt {{t_{{1}}}^{10}-4\,{t_{{2}}}^{3}t_{{1}}-
4\,t_{{1}}}-{t_{{1}}}^{5}\right)}{2\,{t_{{2}}}^{3}+2}},t_{{1}}\right)
$$
 of a surface in $\C^3$, where the root denotes the principal branch. An expression for this is
\[
\renewcommand{\arraystretch}{1.2}
\begin{array}{lll}
	\mP & = & \left\{\left( t_2,{\dfrac { \left( -{t_1}^5+\delta \right) t_2}{2\,{t_
2}^3+2}},t_1\right), \right. \\
	& & \ \mathbb{T} := [\C(t_1,t_2)\subset\C(t_1,t_2)(\delta), \delta^2={t_1}^{10}-4\,{t_2}^3t_1-4\,t_1], \\
	& & \left. \delta(1,-1)=+1 \right\}.
\end{array}
\]	
The polynomials in (\ref{EQ: AP unprojected}) are
\[
\begin{array}{lll}
E_1 &= &-{T_1}^{10}+4\,{T_2}^3T_1+\Delta^2+4\,T_1 \\
G_1 &= & X_1-T_2 \\
G_2 &= & X_2 \left(2\,{T_2}^3+2\right) - \left( -{T_1}^5+\Delta \right) T_2 \\
G_3 &=& X_3-T_1 \\
G_Z &=& Z \left( 2\,{T_2}^3+2 \right) - 1.
\end{array}
\]
Then:
\begin{itemize}
	\item Eliminating $Z$ in $\{E_1,G_1,G_2,G_3,G_Z\}$ we obtain $\AP=V(\mathcal{F})$ where
	\begin{equation}\label{EQ:impl-Ex-Superficie}
	\begin{array}{lll}
	\mathcal{F} & = \{ &  X_1-T_2, \\
	& & X_3-T_1, \\
	& & X_2{T_1}^5+X_2\Delta+2\,T_2T_1, \\
	& & {T_1}^{10}-4\,T_1{T_2}^3-\Delta^2-4\,T_1, \\
	& & {X_2}^2{T_2}^3-X_2\Delta\,T_2-T_1{T_2}^2+{X_2}^2, \\
	& & T_2{T_1}^5+2\,X_2{T_2}^3-\Delta\,T_2+2\,X_2 \quad \}
	\end{array}
	\end{equation}
	and we conclude that $\AP$ is irreducible of dimension 2.
	\item Since $\AP$ is irreducible, and using that $\dim(\BP)=\dim(\AP)$ (this fact will be proved later in Theorem \ref{Th:dimension}, but for illustrating purposes we use it here), we have that $\BP=\AP$.
	\item $\VP$ is the projection of $\BP$ onto the variables $X_1,X_2,X_3$ which is the surface defined by the polynomial $X_1{X_3}^5X_2+{X_1}^3{X_2}^2+{X_1}^2X_3+{X_2}^2$.
\end{itemize}
\end{example}

\begin{example}\label{EX: two parabolas}
Let $\mP=(\sqrt[4]{t^2},t)$ where $\sqrt[4]1=+1$. We have for instance the expression
\[
\mP = \left\{\left(\delta,t\right), \ \mathbb{T} := [\C(t)\subset\C(t)(\delta), \delta^4=t^2], \ \delta(1)=+1 \right\}.
\]
Note that the tower is an extension of degree 2.

Then $\AP=V(\Delta^4-T^2,X-\Delta,Y-T)$, which has two irreducible components; one of them is $\BP$, namely $V(\Delta^2-T,X-\Delta,Y-T)$, which is projected on $\VP=V(X^2-Y)$.

We can say more: the whole $\AP$ is projected on $V(X^4-Y^2)$ which is the union of the parabolas $X^2=\pm Y$. The image of $\mP$ is one ``half" of the parabola $X^2=Y$, and $\VP$ is the whole parabola. The ``conjugate parametrizations'' $(\ii^k\sqrt[4]{t^2},t)$, $k=1,2,3$, cover the other three half-parabolas.
\end{example}

As the previous example shows, $\AP$ contains points related to the different conjugates of the radicals, since the defining polynomials $E_i$ (see (\ref{EQ: AP unprojected})) do not discriminate them. In other words, the same variety $\AP$ is obtained if we conjugate any radicals in the parametrization.

In the following theorem we collect some basic facts that will be used throughout the paper.

\begin{theorem}\label{TH:resumen} \
\begin{enumerate}
\item Let  $V_1,V_2$ be distinct irreducible components of an algebraic variety $V\subset \C^n$. Then, the points in $V_1\cap V_2$ are singular points of $V$.
\item Let $V = V (f_1,\ldots, f_r)\subset \C^n$ be an algebraic variety and  $ p\in V$ such that the Jacobian $J_p(f_1,\ldots,f_r)$ has rank $r$. Then $p\in V$ is a nonsingular point of $V$ and lies on a unique irreducible component of $V$ of dimension $n-r.$
\item The dimension of an algebraic variety $V\subset \C^n$ is the largest dimension of a subspace $H\subset \C^n$ for which a projection of $V$ onto $H$ is Zariski dense.
\item Let $V$ be an irreducible quasi-projective variety and $\Pi:V\rightarrow \mathbb{P}^n(\C)$ be a regular map. Then
\[
\dim(V)=\dim(\Pi(V))+\min\{\dim(\Pi^{-1}(q))\,|\, q\in \Pi(V)\}.
\]
\item Let $\mathcal{F}$ be a finite set of complex polynomials in $n$ variables such that $V(\mathcal{F})\neq \emptyset$. Every irreducible component of $V(\mathcal{F})$ has dimension $\geq n-|\mathcal{F}|$.
\end{enumerate}
\end{theorem}

\begin{proof}
(i), (ii) and (iii) are stated in \cite{CoxLittleOshea2007a}, Chapter 9.6, Theorem 8(iv), page 490; Chapter 9.6, Theorem 9, page 492; and Chapter 9.5, Proposition 5, page 480, respectively. (iv) is stated in \cite{Harris1992a}, Corollary 11.13, page 139. (v) is a direct consequence of Proposition I.7.1. in \cite{Hartshorne1977a}, page 48; see also Proposition 6.1.9 in \cite{Wang2001a}. 
\end{proof}

Next we will prove that $\BP$ is an irreducible component of $\AP$.

\begin{theorem}\label{TH: im(phi) in only one component of BP}
Let $\mP$ be a radical parametrization. Then $\BP$ is contained in a unique irreducible component of $\AP$.
\end{theorem}

\begin{proof}
We will prove that $\Ima(\varphi)$ is contained in one component of $\AP$ by contradiction. In short, two points in different components can be connected by a path so that all its points  are regular. But by continuity the path must contain a point that lies in more than one component, thus singular (see  Theorem \ref{TH:resumen} (i)).
We develop the proof below. We assume w.l.o.g. that no $\delta_i$ is $0$; note that if $\delta_i=0$ then $\F_i=\F_{i-1}$ and it can be omitted.

We introduce some notation. Let $\DP=V(E_1,\ldots,E_m,G_1,\ldots,G_r,G_Z)$, and let $D\subset \C^n$ consisting of those $\ot\in \C^n$ such that
\begin{itemize}
\item for some $i=1,\ldots,m$, either $\alpha_{iN}(\ot,\odelta(\ot))=0$ or $\alpha_{iD}(\ot,\odelta(\ot))=0$,
\item for some $i=1,\ldots,r$, \ $x_{iD}(\ot,\odelta(\ot))=0.$
\end{itemize}
Let $\Omega=\C^n\menos D$, and let
\[
\pi_Z\colon \DP \subset \C^{n+m+r+1}\longrightarrow \AP \subset \C^{n+m+r} \colon (\ot,\odelta,\ox,Z) \mapsto (\ot,\odelta,\ox), \]
\[ \varphi_Z=\pi_{Z}^{-1}\circ \varphi: \Omega\subset \C^{n} \longrightarrow
\DP \subset \C^{n+m+r+1}.
  \]
Below, in Step 3, we will see that $\overline{\Omega}=\C^n$. So,  since we consider $\varphi$ restricted to $\Omega$, we can reason with  $\Ima(\varphi)$ instead of $\BP$. We assume, by contradiction, that there exist two different irreducible components $\Gamma_1,\Gamma_2$ of $\AP$ containing $\Ima(\varphi)$. Let $\Sigma_1,\Sigma_2$ be the two irreducible components of $\DP$ that project, via $\pi_Z$, onto $\Gamma_1,\Gamma_2$, respectively; note that $\pi_Z$ is 1:1.

\medskip

\noindent \textbf{Step 1.} Let $\ot_0\in \Omega$. Then  $p:=\varphi_Z(\ot_0)$ is nonsingular in $\DP$, and lies in a unique component (which has dimension $n$).

\medskip

\noindent For the proof of this step we apply Theorem \ref{TH:resumen} (ii) to the point $p$ and the variety $\DP$. For this it suffices to find a nonzero minor of order $m+r+1$ in the Jacobian matrix of $\DP$. Consider the submatrix corresponding to the partial derivatives with respect to all the variables except the $\OT$. It is lower triangular, and the diagonal elements are $\frac{\partial E_i}{\partial \Delta_i}=e_i(\Delta_i)^{e_i-1}\cdot\alpha_{iD}(\OT,\Delta_1,\ldots,\Delta_{i-1})$ and denominators occurring in the definition of the parametrization. Upon evaluation, the latter are nonzero by construction and the former are nonzero by hypothesis. Thus, $p$ is nonsingular and lies in a unique component of $\DP$ of dimension $n$.

\medskip

\noindent \textbf{Step 2.} Let $W\subset \C^n$ be algebraic of dimension $\leq n-1$. Then $\C^n\menos W$ is path connected. We prove it by induction on $n$:
\begin{itemize}
\item $n=1$: $\C$ is homeomorphic to $\mathbb{R}^2$, and $W$ is a finite or empty set.
\item Induction step: Let $W\subset \C^{n+1}$  be algebraic of dimension $\leq n$. Let $p,q\not \in W$, we will prove that there is a path between them disjoint from $W$. Let $\Pi$ be a hyperplane containing $p$ and $q$. Now, $\C^{n+1}\cap\Pi\cong \C^n$ and $W\cap\Pi$ is algebraic of dimension $<n$ since $\Pi\not\subset W$. By induction there is a path between $p$ and $q$ contained in $\Pi$ and disjoint from $W\cap\Pi$, thus also disjoint from $W$.
\end{itemize}

\noindent \textbf{Step 3.}  $D\subset \C^n$ is contained in an algebraic set of dimension $<n$:
\begin{itemize}
\item Let $\beta\in\F_i\menos\{0\}$. Since $\beta$ is algebraic over $\F_0$, let its minimal polynomial be, after clearing denominators,
\[
q_\ell(\ot)Z^\ell+q_{\ell-1}(\ot)Z^{\ell-1}+\cdots+q_0(\ot), \qquad q_j\in \C[\ot][Z], \quad q_0\neq0.
\]
Then, from $\beta\left(q_\ell\beta^{\ell-1}+\cdots\right)=-q_0$ we deduce that if $\beta(\ot)=0$ then $q_0(\ot)=0$. But the zeroset of $q_0$ is algebraic of dimension $<n$.
\item Since $D$ is contained in the union of the zerosets of the numerators and denominators of the $\delta_i^{e_i}\in\F_{i-1}\menos \{0\}$, the claim is proven.
 \end{itemize}

\noindent \textbf{Step 4.} Let $W\subset \C^n$ be algebraic of dimension $\leq n-1$ such that $D\subset W$ (see Step 3). Then, $\Omega':=\C^{n}\menos W$ is path connected (see step 2), it is included in $\Omega$, and it is dense in $\C^{n}$. So we can continue our proof by reasoning on the image of $\varphi$ restricted to $\Omega'$. But first, we remind the reader that we are assuming (by contradiction) that $\Ima(\varphi)\subset \C^{n+m+r}$ is contained in two components  $\Gamma_1$ and $\Gamma_2$ of $\AP$, and that $\Sigma_1$ and $\Sigma_2$ are the components of $\DP$ that are projected, via $\Pi_Z$, over $\Gamma_1$ and $\Gamma_2$ respectively. Now suppose that $\ot_1,\ot_2\in \Omega'$, and that $p=\varphi_Z(\ot_1)\in \Sigma_1$ and $q=\varphi_Z(\ot_2)\in \Sigma_2$ lie in different components of $\DP$ (note that by Step 1 each one lies in only one component):
\begin{itemize}
\item Consider a path between $\ot_1$ and $\ot_2$ in $\Omega'$ and its image $\gamma\subset\DP$ by $\varphi_Z$ ($\gamma$ is connected because $\varphi_Z$ is continuous in the ambient topology). By Step 1 all the points of $\gamma$ are regular. However we prove now that $\gamma$ contains a singular point: let $C_1$ be the Zariski component of $\DP$ containing $p$ and $C_2$ be the union of the other components (so $q\in C_2$). Then $\gamma\cap C_1$ is closed in $\gamma$ because $C_1$ is closed in the ambient topology, and $\gamma\cap C_2$ is closed likewise. Since they cover $\gamma$, which is connected, and both are nonempty, they are not disjoint. But any point in $C_1\cap C_2$ is singular.
\end{itemize}
\end{proof}

The next result shows that we are indeed working with algebraic sets of dimension $n$.

\begin{theorem}\label{Th:dimension} \
\begin{enumerate}
\item Every component of $\AP$ has dimension $n$.
\item $\BP$ is irreducible and has dimension $n$.
\item $\VP$ is irreducible and has dimension $n$.
\end{enumerate}
\end{theorem}

\begin{proof}
Let $\DP$ be as in the proof of Theorem \ref{TH: im(phi) in only one component of BP}. Consider the projection
\[
\DP\subset \C^{n+m+r+1} \longrightarrow \C^n\colon (\ot,\odelta,\ox,Z) \mapsto \ot.
\]
Every of its fibres is finite, by inspection of the defining system of $\DP$: by the polynomial $G_Z$, all denominators are nonzero, so every polynomial $E_i$ has finitely many solutions in $\Delta_i$, and the same holds for the polynomials $G_j$ and $G_Z$ w.r.t. the variables $X_j$ and $Z$, respectively. Thus by Theorem \ref{TH:resumen} (iv)
every component of $\DP$ has dimension $\leq\dim(\C^n)=n$. On the other hand, by Theorem \ref{TH:resumen} (v),
every component of $\DP$ has dimension $\geq (n+m+r+1)-(m+r+1)=n$. Therefore $\dim(\DP)=n$ and the same is true of its components.
\begin{enumerate}
	\item Let $\pi_Z$ be as in the proof of Theorem \ref{TH: im(phi) in only one component of BP}. Clearly $\pi_Z$ is a finite map. So, using that all the components of $\DP$ have dimension $n$, and Theorem \ref{TH:resumen} (iv)
again, one gets that every component of $\AP$ has dimension $n$.
	\item By Theorem \ref{TH: im(phi) in only one component of BP} and the previous item, we have that $\dim(\BP)\leq n$. On the other hand, let $\Omega$ be as in the proof of Theorem \ref{TH: im(phi) in only one component of BP}, and $\pi^{\times}\colon\C^{n+m+r} \rightarrow \C^n$ such that $\pi^{\times}(\ot,\odelta,\ox)=\ot$. Then, $\Omega\subset \pi^{\times}(\BP)\subset \C^n$. Thus, since $\Omega$ is dense in $\C^n$, it follows that $\pi^{\times}(\BP)$ is also dense in $\C^n$. Therefore, by Theorem \ref{TH:resumen} (iii),
    $\BP$ is one of the irreducible components of $\AP$, and it is of dimension $n$.
    \item Let $W=\overline{\Ima(\mP)}$. There exists a dense subset of the domain of $\mP$ where the rank of the Jacobian of $\mP$ is $n$. Therefore, by Theorem \ref{TH:resumen} (ii), those points are nonsingular in $W$ so $\dim(W)=n$. On the other hand $W\subset\VP$ which is irreducible (because it is the image of the projection from $\BP$). Therefore $W=\VP$.
\end{enumerate}
\end{proof}

\begin{remark}
By inspecting the defining equations of $\DP$ (see the proof of Theorem \ref{TH: im(phi) in only one component of BP}) one can check that each fiber of the map\linebreak $\DP \rightarrow \C^n\colon (\ot,\odelta,\ox,Z)\mapsto \ot$ has cardinality $\leq\prod e_i$. Therefore, this is a bound for the number of components of $\DP$, and hence of $\AP$.
\end{remark}

\begin{remark}
$\VP$ coincides with the intuitive definition of radical variety as $\overline{\Ima(\mP)}$, since the former is irreducible and due to the Jacobian rank condition. Therefore, $\VP$ does not depend on the construction above or the choice of the expression \eqref{ratfunsexpr}.
\end{remark}

\begin{remark} [Implicitization algorithm for $\BP$ and $\VP$] \label{REM: implicit BP and VP}
The following algorithm computes equations of $\BP$ and $\VP$:
\begin{enumerate}
\item Write equations for $\AP$ by eliminating $Z$ from \eqref{EQ: AP unprojected}.
\item Decompose $\AP$ and evaluate $\varphi$ at any parameter value to determine a point inside $\BP$.
\item Eliminate all the variables except the $\OX$ to obtain the equations of $\VP$.
\end{enumerate}
Example \ref{ex:two parabolas second part} shows that there could be a relationship between the conjugations of the radical parametrization and the irreducible components of $\AP$. Nevertheless, we do not address this question in this article.
\end{remark}

\begin{example}\label{ex:two parabolas second part} [Example \ref{EX: two parabolas} continued]
We calculated before that\linebreak $\AP=V(\Delta^4-T^2,X-\Delta,Y-T)$ which decomposes as
\[
\AP=V(\Delta^2-T,X-\Delta,Y-T)\ \cup\ V(\Delta^2+T,X-\Delta,Y-T).
\]
Call the components $W_1$ and $W_2$ respectively.
Since $\varphi(1)=(1,1,1,1)\in W_1\menos W_2$, it follows that $\BP=W_1$. Eliminating in $\BP$ the variables $\Delta,T$ we obtain that $\VP=V(X^2-Y)$.
\end{example}

\section{Rationality and reparametrization of radical varieties}\label{SECTION: rational vs radical}

One of the first questions we can ask is whether a given radical parametrization is actually parametrizing a rational variety. In this section we discuss this question.
For this purpose, we will introduce a variety, called the tower variety, that will provide the necessary information for our goal. In addition, we will also
define the notion of tracing index that will allow us to count algebraically the fibers of a radical parametrization, and that we motivate in the next example.

\begin{example}
Consider the radical curve parametrization $\mP=(t^2,\sqrt{t})$ where $\sqrt{1}=+1$. That is, with the terminology introduced in Section \ref{SECTION: notion of radparam},
\[
	\mP=\left\{(t^2,\delta), \mathbb{T}:=[\C(t)\subset \C(t)(\delta), \delta^2=t ], \delta(1)=1 \right\}.
\]
Then $\varphi(t)=(t,\sqrt{t},t^2,\sqrt{t})$ (see (\ref{eq:diagram 1})). The incidence variety $\BP$ is defined by the equations $\Delta^2=T, X=T^2, Y=\Delta$ and it is irreducible.
The radical variety  $\VP$ is the zeroset of $X=Y^4$.  We study now the fiber of $\pi$ and $\mP$ for different points in $\VP$ (see (\ref{eq:diagram 1})). Consider the two points $p=(16,2)$ and $q=(16,-2)$ in $\VP$:
\begin{itemize}
\item $\pi^{-1}(p)=\{(4,2,p)\}$ and $\pi^{-1}(q)=\{(4,-2,q)\}$. In fact every fiber of $\pi$ has exactly one point, since the second coordinate determines the value of $t$ uniquely, and this gives precisely one point in $\BP$.
\item However, $p$ has a unique preimage by $\mP$, namely $t=4$, but $q$ has no preimage; $\mP^{-1}(p)=\{4\}$, and $\mP^{-1}(q)=\emptyset$.
\end{itemize}
Indeed, the parametrization covers ``half" of $\VP$. Analogously, $\varphi$ covers half of $\BP$ (for example $\pi^{-1}(q)$ is not in the image of $\varphi$).

This example shows how in some sense the fibers of $\pi$, from $\VP$ to $\BP$, behave better than those of $\mP$, from $\VP$ to $\C$.
\end{example}

\begin{definition}
The \emph{tracing index} of a parametrization $\mP$ is the degree of the map $\pi$, that is, the generic cardinal of $\pi^{-1}(p)$ for $p\in\VP$.
\end{definition}

\begin{example}
Consider $\mP=(t^2,\sqrt{t^2+1})$; more precisely, $\mP=\{(t^2,\delta),\mathbb{T}:=[\C(t)\subset \C(t)(\delta), \delta^2=t^2+1], \delta(1)=+\sqrt{2}\}$. Then $\BP=V(\Delta^2-(T^2+1),X-T^2,Y-\Delta)$ and $\VP=V(Y^2-(X+1))$. Every point $(a,b)\in\VP$ with $a\neq0$ has two preimages in $\BP$ by $\pi$, namely
$\pi^{-1}(a,b)=\{(\pm\sqrt{a},b,a,b)\}$. Therefore the tracing index of $\mP$ is two. On the other hand, $(1,\sqrt2)\in\VP$ has as preimages by $\mP$ the values $t=\pm1$, but $(1,-\sqrt2)\in\VP$ has no preimage (the chosen branch produces positive roots of positive real numbers). Indeed, only half of the points of $\VP$ are covered by $\mP$.
\end{example}

In the next example we illustrate how to compute the tracing index.

\begin{example} [Example \ref{EX: superficie} continued] \label{EX: compute tracing index}
We consider the radical surface parametrization $\mP$ of Example~\ref{EX: superficie}. One possibility for computing the tracing index is to proceed probabilistically; that is, we take a random point in $p\in\VP$ and then determine the cardinality of $\pi^{-1}(p)$. For this purpose, we may give values to the parameters in $\mP$. Let us take
\[ p=\mP(4,8)=\left( 8, \dfrac{112}{513}\sqrt{1327}-\dfrac{4096}{513}, 4\right). \]
The fiber of $p$ by $\pi$ is
\[ \pi^{-1}(p)=\left\{
\left( 4, 8, 28\sqrt{1327}, 8, \dfrac{112}{513}\sqrt{1327}-\dfrac{4096}{513}, 4 \right)
\right\}. \]
Therefore, the tracing index of $\mP$ is 1. Alternatively, if one wants to compute deterministically the tracing index, we may repeat the above computation with a generic point in $\VP$. More precisely, let $(a,b,c)\in \C^3$ such that $F(a,b,c)=0$ where $F(X_1,X_2,X_3)$ is the defining polynomial of $\VP$ (see Example~\ref{EX: superficie}). In this situation, we consider the polynomials $$\mathcal{G}:=\{ F(a,b,c),X_1-a,X_2-b,X_3-c\}\cup \mathcal{F}$$ where $\mathcal{F}$ is the set of generators of $\BP$ (see (\ref{EQ:impl-Ex-Superficie})). We compute the reduced Gr\"obner basis of $\mathcal{G}$ with respect to the pure lexicographic order with $a>b>c>X_1>X_2>X_3>t_1>t_2>\Delta$, obtaining
\[\begin{array}{ll}
	\{ & {t_1}^{10}-4\,t_1{t_2}^3-\Delta^2-4\,t_1,\ c-t_1,\ {t_1}^5t_2+2\,b{t_2}^3-\Delta\,t_2+2\,b, \\
	& b{t_1}^5+\Delta\,b+2\,t_2t_1,\ a-t_2,\ X_3-t_1,\ X_2-b,\ X_1-t_2 \quad \}.
\end{array} \]
Then,
\[
	\pi^{-1}(a,b,c)=\left\{ \left((c,a),-\dfrac {c \left( b{c}^{4}+2\,a \right) }{b},(a,b,c)\right)\right\}.
\]
Therefore, the tracing index is 1, and the inverse map is
\[ \begin{array}{cccc}
\pi^{-1}: & \VP & \longrightarrow & \BP \\[1ex]
          & \OX & \longmapsto     &  \left((X_3,X_1),-\dfrac {X_3 \left( X_2{X_3}^{4}+2\,X_1 \right) }{X_2},\OX \right).
\end{array} \]
\end{example}

We introduce now our next concept. For this purpose, we define the map
\begin{equation}\label{EQ:psi}
\psi\colon\C^n\to\C^{n+m}\colon\ot\mapsto(\ot,\odelta).
\end{equation}
\begin{definition}
We define the \emph{tower variety} of $\mP$ as $\VT=\overline{\Ima(\psi)}$, and the rational maps
\[
\mR\colon\VT\to\VP\colon(\ot,\odelta)\mapsto\ox \quad\mbox{and}\quad \pi^*\colon\BP\to\VT\colon(\ot,\odelta,\ox)\mapsto(\ot,\odelta).
\]
\end{definition}

Therefore $\mR$ is a rational lift of the nonrational $\mP$. The next diagram shows the previous and new definitions.
\begin{equation}\label{eq:diagrama2}
\hspace*{-3mm}
\xy
(0,0)*{\xy
	(-5,5)*++{\BP}="BP";
	"BP"+(19,0.5)*{\subset\AP\subset\C^{n+m+r}};
	(-11,-25)*++{\VP}="VP";
	"VP"+(-5,0)*{\supset}; "VP"+(-9.5,0)*++{\C^r};		
	(20,-25)*++{\C^n}="Cn";
	{\ar_{\displaystyle \pi} "BP"; "VP"};
	{\ar^(0.6){\displaystyle \varphi} "Cn"; "BP"};
	{\ar_(0.4){\displaystyle \mP} "Cn"; "VP"};
	(22,-10)*++{\VT}="VT";
	"VT"+(5,0)*{\subset}; "VT"+(12,0.5)*++{\C^{n+m}};
	{\ar_{\displaystyle \psi} "Cn"; "VT"};
	{\ar_(0.6){\displaystyle \mR} "VT"; "VP"};
	{\ar^{\displaystyle \pi^*} "BP"; "VT"};
\endxy};
(35,-8)*{\mbox{defined as}};
(70,0)*{\xy
	(-5,5)*+{(\ot,\odelta,\ox)}="BP";
	(-11,-25)*++{\ox}="VP";
	(20,-25)*++{\ot}="Cn";
	{\ar@{|->}_{\displaystyle \pi} "BP"; "VP"};
	{\ar@{|->}^(0.6){\displaystyle \varphi} "Cn"; "BP"};
	{\ar@{|->}_(0.4){\displaystyle \mP} "Cn"; "VP"};
	(22,-10)*+{(\ot,\odelta)}="VT";
	{\ar@{|->}_{\displaystyle \psi} "Cn"; "VT"};
	{\ar@{|->}_(0.6){\displaystyle \mR} "VT"; "VP"};
	{\ar@{|->}^{\displaystyle \pi^*} "BP"; "VT"};
\endxy};
\endxy
\end{equation}

The next theorem states the main properties of the tower variety.

\begin{theorem}\label{th:VT}
$\VT$ is irreducible and has dimension $n$. Furthermore, $\VT=\overline{\pi^{*}(\BP)}$.
\end{theorem}

\begin{proof}
First, we have that $\dim(\VT)=n$ by applying to $\psi$ a similar argument to that of Theorem \ref{Th:dimension} (iii).

Second, by Theorem \ref{Th:dimension}, $\dim(\BP)=n$. So $\pi^*(\BP)$ has dimension $n$, because its fibers are generically 0-dimensional (see Theorem \ref{TH:resumen} (iv)).
In order to show that $\VT\subset\overline{\pi^*(\BP)}$, it suffices to prove the inclusion of the dense subset $\psi(\Dom(\varphi))$, which is true by the commutativity of the corresponding diagram. Therefore $n\geq\dim(\VT)\geq n$.

Finally, since $\BP$ is irreducible (see Theorem~\ref{Th:dimension}) and $\pi^*$ is finite, we have that $\pi^*(\BP)$ is irreducible. Therefore, since $\VT\subset \overline{\pi^*(\BP)}$, and $\dim(\VT)=n=\dim(\pi^{*}(\BP))$, we have that $\VT=\overline{\pi^{*}(\BP)}$ and hence it is irreducible.
\end{proof}

Theorem \ref{th:VT} provides an algorithm to determine the tower variety.

\begin{remark} [Implicitization algorithm for $\VT$] \label{REM: implicit VT}
The following algorithm computes equations of $\VT$:
\begin{enumerate}
\item Compute  defining polynomials of $\BP$ (see Remark \ref{REM: implicit BP and VP}).
\item Eliminate the variables $\OX$ to obtain the equations of $\VT$.
\end{enumerate}
\end{remark}

\begin{example} [Example \ref{EX: compute tracing index} continued] \label{EX: compute tower}
We consider the radical parametrization $\mP$ introduced in Example \ref{EX: superficie}. There we computed generators of $\BP$, namely the set $\mathcal{F}$. In this situation, the tower variety can be determined by eliminating in $\mathcal{F}$ the variables $\OX$. We obtain $\VT=V(-T_{1}^{10}+4T_1 T_{2}^{3}+\Delta^2+4T_1) \subset \C^3$.
\end{example}

The importance of $\VT$ and $\mR$ resides in the fact that they encode rationally the information of the radical parametrization. For example, for curves, the mere existence of $\mR$ implies that $\genus(\VP)\leq\genus(\VT)$; in particular, if $\VT$ is rational, then $\VP$ is rational. In other words:

\begin{theorem}
Let $n=1$. Given a tower of fields such that $\VT$ is rational, any radical parametrization from that tower will give rise to a rational curve.
\end{theorem}

There is no general bound on the discrepancy between $\genus(\VT)$ and $\genus(\VP)$, as the next example shows.
\begin{example}
The parametrization $\mP(t)=\left(\sqrt[n]{1-t^n},\sqrt[n]{1-t^n}\right)$ is $n:1$ and $\VP$ is the line $X_1=X_2$, clearly a genus 0 curve. But for the tower $\C(t)\subset \C(t)(\sqrt[n]{1-t^n})$ we have $\VT=V(T^n+\Delta^n-1)$, a Fermat curve with genus $\frac{(n-1)(n-2)}{2}$.
\end{example}

For general dimension, if $\VT$ is unirational then $\VP$ is unirational. Similar conditions on the plurigenera may exist in higher dimension.

\begin{theorem}\label{theorem-rational}
$ $
\begin{enumerate}
\item If the tracing index of $\mP$ is 1, then $\VP$ and $\VT$ are birationally equivalent. Therefore, $\VP$ is rational if and only $\VT$ is rational.
\item If $\VT$ is unirational, then $\VP$ is unirational. Furthermore, if
$\mathcal{T}(\ot)$ is a unirational parametrization of $\VT$ then $\mR(\mathcal{T})(\ot)=\mP(\pi_\mathbb{T}\circ\mathcal{T}(\ot))$ is a unirational parametrization of $\VP$ where $\pi_\mathbb{T}(\ot,\overline{\delta})=(\ot)$.
\end{enumerate}
\end{theorem}

\begin{proof}
$ $
\begin{enumerate}
	\item Since $\mR=\pi\circ(\pi^*)^{-1}$ over irreducible dense sets, and both are injective, $\mR$ is rational and injective, thus birational (see also \cite{Schicho1998a} for a constructive proof).
	\item Clear since $\mP\circ\pi_\mathbb{T}=\mR$.
	\end{enumerate}
\end{proof}

The last part of the theorem gives rise to an algorithm for rational reparametrization. Assume that an algorithm that decides the existence of rational parametrizations and computes them is available; see for instance \cite{SendraWinklerPerezdiaz2008a} for curves and \cite{Schicho1998b} for surfaces. Call this algorithm \textsf{RatParamAlg}. Then, the last two theorems provide a reparametrization algorithm.

\begin{algorithm}[Reparametrization Algorithm]
\renewcommand{\algorithmicrequire}{\textbf{Input:}\ }
\renewcommand{\algorithmicensure}{\textbf{Output:}\ }
\begin{algorithmic}[1]
$ $\vspace{1ex}

\algorithmicrequire A radical parametrization $\mP$ given as in Def \ref{DEF: rad param}.

\algorithmicensure One of the following:
\begin{itemize}
 \item a reparametrization of $\mP$ that makes it rational
 \item ``$\VP$ cannot be parametrized rationally''
 \item ``No answer''.
\end{itemize}
\STATE Compute $\VT$ and apply \textsf{RatParamAlg} to it.
\IF {$\VT$ has a rational parametrization $\mT(\ot)$}
	\RETURN $\pi_\mathbb{T}(\mT(\ot))$ with $\pi_\mathbb{T}(\OT,\overline{\Delta})=(\OT)$
\ELSE
	\STATE compute the tracing index $m$ of $\mP$
	\IF {$m=1$}
		\RETURN ``$\VP$ cannot be parametrized rationally''
	\ELSE
		\RETURN ``No answer''
	\ENDIF
\ENDIF
\end{algorithmic}
\end{algorithm}

\begin{remark}
When the algorithm outputs ``No answer"  it means that it was unable to find out whether $\mP$ can be reparametrized into a rational parametrization. This is due to the fact that the tracing index is not 1. Nevertheless, this can be solved by applying \textsf{RatParamAlg} to the radical variety $\VP$, but the idea of the algorithm is to provide the answer from the tower variety.
\end{remark}


The following examples illustrate the algorithm.

\begin{example} \label{EX: reparam circle cubicroot}
	Let $\mP=(\sqrt[3]{t},\sqrt{1-\sqrt[3]{t^2}})$, that is,
	\[
	\begin{array}{ll}\mP:= & \{ (\delta_1,\delta_2), \mathbb{T}:=[ \C(t)\subset \C(t)(\delta_1)\subset \C(t,\delta_1)(\delta_2), \delta_1^3=t, \delta_2^2=1-\delta_1^2 ], \\
	& \delta_1(1)=1, \delta_2(1)=0.  \}
	\end{array} \]
	We have $\VT=V(\Delta_1^3-T,\Delta_2^2-(1-\Delta_1^2))$, that can be parametrized by $$\mT(t)=\left( \left(\frac{2t}{1+t^2}\right)^3,\frac{2t}{1+t^2},\frac{1-t^2}{1+t^2} \right).$$ Then we obtain the rational reparametrization $$\mP\left(\left(\frac{2t}{1+t^2}\right)^3\right) = \left(   \frac{2t}{1+t^2},\frac{1-t^2}{1+t^2} \right).$$
\end{example}

\begin{example} \label{EX: reparam superficie}
We consider the surface defined by $x^3-x^2z-xz^2+z^3-8y^2=0$. Let us say that we want to parametrize it, if possible, rationally. One standard way is to deduce it from the computation of the arithmetic genus and the plurigenus of the surface. Instead, since the degree w.r.t. $y$ of the defining polynomial is 2, we may first compute a radical parametrization (see \cite{SendraSevilla2013a} for further details). We obtain, for instance,
\[ \renewcommand{\arraystretch}{1.5} \begin{array}{l}
	\mP=\left( t_1, \frac{1}{4}\sqrt{2t_1+2 t_2}\, (-t_2+t_1), t_2\right) \\
	=\left\{ \left(t_1, \frac{1}{4} \delta (t_1-t_2), t_2\right), \mathbb{T}:=[\C(\ot)\subset \C(\ot)(\delta),
\delta^2=2(t_1+t_2) ], \delta(1,1)=2   \right\}.
\end{array}  \]
The tower variety is
\[ \VT=V(\Delta^2-2T_1-2T_2) \]
which is clearly rational and parametrizable as
\[ \mathcal{T}(h_1,h_2)=\left( \frac{1}{2} h_{1}^{2}-h_2, h_2, h_1\right). \]
Now, we can reparametrize:
\[ \mP\left( \frac{1}{2} h_{1}^{2}-h_2, h_2\right)=\left(\frac{1}{2}h_{1}^{2}-h_{2}, \frac{1}{4}h_1 \left(-2h_2+\frac{1}{2}h_{1}^{2}\right), h_2\right)\]
which is a rational parametrization of the given surface.
\end{example}

Even if we cannot reparametrize a radical variety rationally, sometimes we may be able to simplify the radical parametrization using the same technique.

\begin{example} \label{EX: reparam simplificacion}
	Suppose $\VT=V(\Delta_1^2-T, \Delta_2^3-(12-T))$. This is a genus 1 curve, thus it admits a $2:1$ rational map to $\C$. Such map can be inverted by means of a square root, that is, there exists $\mT_2\colon\C\rightarrow\VT$ radical and defined in an extension of $\C(t)$ of degree 2:
	\[ \mathcal{T}_2(t)=(-t^3+12, \sqrt{-t^3+12}, t). \]
	Now, given any $\mP$ with the tower $\C(t)\subset \C(\sqrt{t})\subset \C(\sqrt{t},\sqrt[3]{12-t^2})$, we have that $\mR\circ\mT_2$ is a parametrization of $\VP$ defined in the same degree 2 extension as $\mT_2$. In other words, $\pi_\mathbb{T}\circ\mT_2$ is a reparametrization that simplifies $\mP$.
\end{example}

\subsection{Parametrizing surfaces}
Example \ref{EX: reparam superficie} motivates the fact that the tower variety can be used to reparametrize rationally surfaces that are parametrized by radicals. The feasibility of this approach depends on whether the tower variety is easier to parametrize rationally than the radical variety. Although we do not have a complete answer to this question, we illustrate how these ideas can be applied to certain families of surfaces for which one does not need to apply the general surface parametrization algorithm but simpler instances of it.

\bigskip

\noindent \textit{Tubular Tower Surfaces.}
We consider irreducible surfaces of the form
\[ F(x,y,z):=A(x,y)^2 z^2 - B(x,y)^2 C(x,y) \]
where $A,B\in \C[x,y]$ and $C\in \C[x,y]$ with degrees at most two w.r.t. one of the variables, say $y$.
The surface can be parametrized by radicals
using the tower
\[ \mathbb{T}:=[\C(\ot)\subset\C(\ot)(\delta),\,\,\text{where}\,\,\delta^2=C(\ot)]  \]
as
\[ \mP(\ot)=\left(t_1,t_2, \dfrac{B(\ot)\delta(\ot)}{A(\ot)}\right). \]
Then, the tower variety is
\[ \VT=V(\Delta^2-C(\OT))\]
which is a tubular surface. Therefore, $\VP$ is rational. Moreover, the tubular surface parametrization algorithm applied to $\VT$ produces the change of parameters to be performed in $\mP(\ot)$ in order to get a rational parametrization of $\VP$. We observe that if the total degree of $C$ is $\leq 2$, the tower variety is indeed a quadric; this is the case of Example \ref{EX: reparam superficie}.

Using an analogous reasoning, the previous family can be extended to
\[ F(x,y,z):=A(x,y)^2 z^2 + A(x,y)^2B(x,y)^2D(x,y)z+ B^2(x,y) C(x,y)\]
where  $A,B,C,D\in \C[x,y]$ and the degree of $D^2-4C$ is at most two w.r.t. one of the variables.

\bigskip

\noindent \textit{Monoidal Tower Surfaces.}
We consider irreducible surfaces of the form
\[ F(x,y,z):=A(x,y)^n z^n - B(x,y)^n C_{n-1}(x,y) \]
where $A,B\in \C[x,y]$ and $C_{n-1}\in \C[x,y]$ is a homogenous form of degree $n-1$.
The surface can be parametrized by radicals
using the tower
\[ \mathbb{T}:=[\C(\ot)\subset\C(\ot)(\delta),\,\,\text{where}\,\,\delta^n=C_{n-1}(\ot)]  \]
as
\[ \mP(\ot)=\left(t_1,t_2, \dfrac{B(\ot)\delta(\ot)}{A(\ot)}\right). \]
Then, the tower variety is
\[ \VT=V(\Delta^n-C_{n-1}(\OT))\]
which has an $(n-1)$-fold point at the origin. Therefore, $\VP$ is rational. Moreover, parametrizing $\VT$ by lines we obtain the change of parameters to be performed in $\mP(\ot)$ in order to get a rational parametrization of $\VP$.

\begin{example}
We consider the surface defined by
\[ F(x,y,z)=-x^{14}y^5+y^8z^4+4xy^6z^4+6x^2y^4z^4+4x^3y^2z^4+x^4z^4\]
which can be expressed as
\[ F(x,y,z)=(y^2+x)^4z^4-(x^3y)^4x^2y \]
and hence it is of monoidal tower type. We consider the tower
\[ \mathbb{T}:=[\C(\ot)\subset\C(\ot)(\delta),\,\,\text{where}\,\,\delta^4=t_{1}^2t_{2}]  \]
that provides the radical parametrization
\[ \mP(\ot)=\left(t_1,t_2, t_{1}^{3} \,t_2\, \dfrac{\sqrt[4]{t_{1}^{2}t_{2}}}{t_{2}^{2}+t_1}\right).
\]
The tower variety is
\[ \VT=V(\Delta^4-T_{1}^2 T_2) \]
that can be parametrized by lines as
\[ \mT=(t_{1}^2, t_{2}^4, t_1t_2). \]
Therefore $\mP(t_{1}^2,t_{2}^4)$ is a rational parametrization of the surface, namely
\[ \left(t_{1}^2, t_{2}^4, \dfrac{t_{1}^7 t_{2}^5}{t_{2}^8+t_{1}^2} \right).\]
\end{example}

\subsection{Radical Integrals}

Finally, rational reparametrization can be applied to the computation of certain radical integrals. More precisely, let us consider the integral
\[ \int f(t)\, dt \]
where $f$ is an element of a radical tower $\mathbb{T}$ over $\C(t)$. We have the tower variety $\VT=\overline{\Ima(\psi)}$ associated to $\mathbb{T}$ (see (\ref{EQ:psi})). We assume that $\VT$ is rational, and let $\mT(t)=(\mT_1(t),\ldots)$ be a birational parametrization of the curve $\VT$. Then, by Theorem \ref{theorem-rational}, $f(\pi_{\mathbb{T}} \circ \mT)\in \C(t)$. We consider the new  parametrization $\mT^*(t):=\psi(\mT_1(t))$  of $\VT$ that, by Theorem \ref{theorem-rational}, is rational. Furthermore, since the first coordinate of both parametrizations $\mT$ and $\mT^*$ are equal, and since $\mT$ is proper, then $\mT^*$ is also birational.  Moreover, since $\xi(t):=\pi_{\mathbb{T}} \circ \mT^*\in \C(t)$ then
\[ \int f(\pi_{\mathbb{T}}(\mT^*(u)))\, \frac{\partial \xi(u)}{\partial u}\, du \]
is a rational integral. Furthermore, since $\mT^*$ is birational, the change of parameter in the integral can be undone as
\[ u= (\mT^*)^{-1}(\psi(t)). \]
We illustrate this application with a couple of examples.

\begin{example}\label{EX: integrales}
Let us show how to compute primitives of the form
\[
\int R\left(\sqrt[3]{t},\sqrt{1-\sqrt[3]{t^2}}\right)\,dt
\]
where $R$ is any rational function. Taking into account Example \ref{EX: reparam circle cubicroot}, we have that
\[ \mT(t)= \left(\dfrac{(2t)^3}{(1+t^2)^3},\dfrac{2t}{1+t^2},\dfrac{1-t^2}{1+t^2}\right)\]
and
\[ \mT^*(t)= \left(\dfrac{(2t)^3}{(1+t^2)^3},\dfrac{2t}{1+t^2},\dfrac{t^2-1}{1+t^2}\right).\]
Now, the integral can be converted, by $t=\left(\frac{2u}{1+u^2}\right)^3$, into the integral
\[\int R\left(\dfrac{2u}{1+u^2},\frac{u^2-1}{1+u^2}\right)\cdot\frac{-24 u^2 (u^2-1)}{(u^2+1)}^4\,du
\]
which has a rational function integrand. Note that
$u=\dfrac{1+\sqrt{-\sqrt[3]{t^2}+1}}{\sqrt[3]{t}}$.
\end{example}

\begin{example}
We illustrate our application to the classical integration method of integrals of type
\[ \mathrm{I}:=\int R(t,\sqrt{at^2+bt+c}) \,dt, \,\,\,\text{where}\,\,\, R\in \C(v_1,v_2) \]
which are typically solved by using the so-called Euler substitutions, see \cite{GradshteynRyzhik2007a}. We consider the radical tower
\[ \C(t)\subset \C(t)(\delta), \,\,\text{where}\,\, \delta^2=at^2+bt+c. \]
The tower variety $\VT$ is the plane curve defined by $\Delta^2=aT^2+bT+c$, a conic and hence rational. A birational parametrization of $\VT$ is
\[ \mT=\mT^*=\left( \frac{u^2-c}{-2 \sqrt{a}\, u+b}, \frac{-\sqrt{a}\, u^2-\sqrt{a}\, c+b u}{-2\sqrt{a}\, u+b}\right). \]
Therefore, the change of parameter is
\[ t=\frac{u^2-c}{-2 \sqrt{a}\, u+b}\]
and
\[ u= \sqrt{at^2+bt+c}-t\sqrt{a}. \]
So the integral turns to be
\[ \mathrm{I}= \int R\left(\dfrac{u^2-c}{-2 \sqrt{a}\, u+b},\dfrac{\sqrt{a}\,u^2+\sqrt{a}\,c-bu}{2\sqrt{a}\,u-b} \right)
\dfrac{-2(\sqrt{a}\,u^2+\sqrt{a}\,c-b u)}{(2\sqrt{a}\,u-b)^2}
\, du \]
which is now a rational integral.
\end{example}


\begin{thebibliography}{19}

\bibitem 
{CavinessFateman1976a}
B.~F. Caviness and R.~J. Fateman.
\newblock Simplification of radical expressions.
\newblock In {\em Proceedings of the Third ACM Symposium on Symbolic and
  Algebraic Computation}, SYMSAC '76, pages 329--338, New York, NY, USA, 1976.
  ACM.

\bibitem 
{CoxLittleOshea2007a}
David Cox, John Little, and Donal O'Shea.
\newblock {\em Ideals, varieties, and algorithms}.
\newblock Undergraduate Texts in Mathematics. Springer, New York, third
  edition, 2007.
\newblock An introduction to computational algebraic geometry and commutative
  algebra.

\bibitem 
{DavenportSiretTournier1988a}
James~H. Davenport, Yvon Siret, and Evelyne Tournier.
\newblock {\em Computer algebra}.
\newblock Academic Press, Inc. [Harcourt Brace Jovanovich, Publishers], London,
  1988.
\newblock Systems and algorithms for algebraic computation, With a preface by
  Daniel Lazard, Translated from the French by A. Davenport and J. H.
  Davenport, With a foreword by Anthony C. Hearn.

\bibitem 
{GradshteynRyzhik2007a}
I.~S. Gradshteyn and I.~M. Ryzhik.
\newblock {\em Table of integrals, series, and products}.
\newblock Elsevier/Academic Press, Amsterdam, seventh edition, 2007.
\newblock Translated from the Russian, Translation edited and with a preface by
  Alan Jeffrey and Daniel Zwillinger, With one CD-ROM (Windows, Macintosh and
  UNIX).

\bibitem 
{Hartshorne1977a}
Robin Hartshorne.
\newblock {\em Algebraic geometry}.
\newblock Springer-Verlag, New York-Heidelberg, 1977.
\newblock Graduate Texts in Mathematics, No. 52.

\bibitem 
{Harris1992a}
Joe Harris.
\newblock {\em Algebraic geometry}, volume 133 of {\em Graduate Texts in
  Mathematics}.
\newblock Springer-Verlag, New York, 1992.
\newblock A first course.

\bibitem 
{Harrison2013a}
Michael Harrison.
\newblock Explicit solution by radicals, gonal maps and plane models of
  algebraic curves of genus 5 or 6.
\newblock {\em J. Symbolic Comput.}, 51:3--21, 2013.

\bibitem 
{HoschekLasser1993a}
Josef Hoschek and Dieter Lasser.
\newblock {\em Fundamentals of computer aided geometric design}.
\newblock A K Peters, Ltd., Wellesley, MA, 1993.
\newblock Translated from the 1992 German edition by Larry L. Schumaker.


\bibitem 
{Lang2002a}
Serge Lang.
\newblock {\em Algebra}, volume 211 of {\em Graduate Texts in Mathematics}.
\newblock Springer-Verlag, New York, third edition, 2002.

\bibitem 
{Schicho1998a}
Josef Schicho.
\newblock Inversion of birational maps with {G}r\"obner bases.
\newblock In {\em Gr\"obner bases and applications ({L}inz, 1998)}, volume 251
  of {\em London Math. Soc. Lecture Note Ser.}, pages 495--503. Cambridge Univ.
  Press, Cambridge, 1998.

\bibitem 
{Schicho1998b}
Josef Schicho.
\newblock Rational parametrization of surfaces.
\newblock {\em J. Symbolic Comput.}, 26(1):1--29, 1998.

\bibitem 
{SendraSevilla2011a}
J.~Rafael Sendra and David Sevilla.
\newblock Radical parametrizations of algebraic curves by adjoint curves.
\newblock {\em J. Symbolic Comput.}, 46(9):1030--1038, 2011.

\bibitem 
{SendraSevilla2013a}
J.~Rafael Sendra and David Sevilla.
\newblock First steps towards radical parametrization of algebraic surfaces.
\newblock {\em Comput. Aided Geom. Design}, 30(4):374--388, 2013.

\bibitem 
{SendraWinkler2001a}
J.~Rafael Sendra and Franz Winkler.
\newblock Tracing index of rational curve parametrizations.
\newblock {\em Comput. Aided Geom. Design}, 18(8):771--795, 2001.

\bibitem 
{SendraWinklerPerezdiaz2008a}
J.~Rafael Sendra, Franz Winkler, and Sonia P{\'e}rez-D{\'{\i}}az.
\newblock {\em Rational algebraic curves}, volume~22 of {\em Algorithms and
  Computation in Mathematics}.
\newblock Springer, Berlin, 2008.
\newblock A computer algebra approach.

\bibitem 
{Wang2001a}
Dongming Wang.
\newblock {\em Elimination methods}.
\newblock Texts and Monographs in Symbolic Computation. Springer-Verlag,
  Vienna, 2001.

\bibitem 
{ZariskiSamuel1975a}
Oscar Zariski and Pierre Samuel.
\newblock {\em Commutative algebra. {V}ol. 1}.
\newblock Springer-Verlag, New York-Heidelberg-Berlin, 1975.
\newblock With the cooperation of I. S. Cohen, Corrected reprinting of the 1958
  edition, Graduate Texts in Mathematics, No. 28.

\end{thebibliography}
\end{document}